\documentclass[12pt,longbibliography]{article}

\usepackage[utf8]{inputenc}

\usepackage[T1]{fontenc}

\usepackage[a4paper, margin=2.7cm]{geometry}

\usepackage{amsmath, amsthm, amsfonts, amssymb}
\usepackage{mathrsfs}   
\usepackage{textcomp}        

\usepackage[hyphens]{url} 
\usepackage[all]{xy}

\usepackage{amscd}
\usepackage[colorlinks=true,linkcolor=blue,citecolor=blue,urlcolor=blue,breaklinks]{hyperref}

\newtheorem{innercustomgeneric}{\customgenericname}
\providecommand{\customgenericname}{}
\newcommand{\newcustomtheorem}[2]{%
	\newenvironment{#1}[1]
	{%
		\renewcommand\customgenericname{#2}%
		\renewcommand\theinnercustomgeneric{##1}%
		\innercustomgeneric
	}
	{\endinnercustomgeneric}
}

\newcustomtheorem{customthm}{Theorem}

\theoremstyle{plain}
\newtheorem{theorem}{Theorem}[section]
\newtheorem{lem}[theorem]{Lemma}
\newtheorem{cor}[theorem]{Corollary}
\newtheorem{prop}[theorem]{Proposition}
\newtheorem{remark}[theorem]{Remark}

\newtheorem*{question*}{Question}
\theoremstyle{definition}
\newtheorem{defn}[theorem]{Definition}

\theoremstyle{remark}

\newcommand{\Q}{\mathbb{Q}}
\newcommand{\Z}{\mathbb{Z}}
\newcommand{\F}{\mathbb{F}}

\hypersetup{pdftitle={Title of the PDF}}
\hypersetup{pdfauthor={Author of the PDF}}

\title{Profinite genus of fundamental groups of compact flat manifolds with the cyclic holonomy group of square-free order}
\author{Genildo de Jesus Nery\thanks{The author held CNPq scholarship during the preparation of this article.}}

\AtEndDocument{\bigskip{\footnotesize
		\noindent
		\textsc{Genildo de Jesus Nery. University of Brasilia, Department of Mathematics, Brasilia DF, Brazil.}
		\textit{E-mail address}: \texttt{\href{mailto:g.j.nery@mat.unb.br}{g.j.nery@mat.unb.br}}
}}

\begin{document}
	
	\maketitle
	
	\begin{abstract}
		In this article we study the extent to which an $n$-dimensional compact flat manifold with the cyclic holonomy group of square-free order may be distinguished  by the finite quotients of its fundamental group. In particular, we display a formula for the cardinality of profinite genus of the fundamental group  of an $n$-dimensional compact flat manifold with the cyclic holonomy group of square-free order. \\
		\
		\\
		\noindent{\bf Keywords:} Profinite genus; Bieberbach group; compact flat manifold. \\
		\
		\\
		\noindent{\bf Mathematics Subject Classification (2010):} 20E18; 20E26; 11R29; 58D17
	\end{abstract}
	
	
	\section{Introduction}
	\label{sec:intro}

	The study of what structural properties of a manifold can be detected by  the set of finite quotients of the fundamental group of a manifold attracted a lot of attention   in 21-st centure (see \cite{BCR16, BMRS18, BMRS20, Wil17, Wil19, WZ-10, WZ-16}).
	
	The $n$-dimensional compact flat manifolds were well described by Bieberbach who characterized such manifolds as isometrically covered by a flat torus such that their fundamental groups $\Gamma$ are torsion-free  with a maximal abelian normal subgroup $M$ (subgroup of all translations) of finite index (see \cite{Cha86}). Consequently, the group $\Gamma$ is called an $n$-dimensional Bieberbach group; the quotient $G=\Gamma/M$, called the holonomy group, is a finite group acting faithfully on $M$. Note that $\Gamma_d$, the full inverse image of a subgroup $H$ of $G$ of order $d$ under the quotient map $\Gamma\rightarrow \Gamma/M$, is an $n$-dimensional Bieberbach group with maximal abelian normal subgroup $M_d$ and  holonomy group $H$.

	In this paper we investigate the extent to which a Bieberbach group may be distinguished from each other by its set of finite quotient groups. Since there is no complete classification of Bieberbach groups in all dimensions, it becomes more difficult to investigate this problem in the general case. Thus, it makes sense to consider this problem for some families. For example, in the previous study \cite{Nery} we give a complete answer to this problem in the case that the holonomy group is cyclic of prime order. In this paper we consider the Bieberbach groups $\Gamma$ with cyclic  holonomy group $G=C_{p_1}\times \cdots \times C_{p_k}$, where $C_{p_i}$ is a cyclic group of prime order $p_i$. Thus, this paper can be considered as a natural continuation of the previous work.

	Following \cite{GZ11}, we define the genus $\mathfrak{g}(\Gamma)$ as the set of isomorphism classes of finitely generated residually finite groups with the profinite completion isomorphic to the profinite completion $\widehat \Gamma$ of $\Gamma$.  In this paper we find a formula for the genus of an $n$-dimensional Bieberbach group with the cyclic holonomy group of square-free order. 
	
	Fix a faithful $\Z G$-lattice $M$ for a finite group $G$.  We define the crystal class $(G,M)$ as the set of all the free-torsion extensions $\Gamma$ of $G$ by $M$. We say that two crystal classes $(G,M)$ and $(G',M')$ are arithmetically equivalent if $G$ and $G'$ are conjugate subgroups of $\mathrm{GL}(n,\Z)$. The resulting equivalence classes are the arithmetic crystal classes. Similar to the definition of  $\mathfrak{g}(\Gamma)$, let $\mathcal{C}( M)$ denote the set of isomorphism classes of $\Z G$-lattice $N$ such that the $\widehat{\Z} G$-modules $\widehat{N}$ and $\widehat{M}$ are isomorphic.
	
	Note that to calculate the cardinality $|\mathfrak{g}(\Gamma)|$ of the genus we have to calculate the number of isomorphism classes of Bieberbach groups in the crystal class $(G,M)$ and $|\mathcal{C}(M)|$. We first calculate $|\mathcal{C}(M)|$. For this purpose, we will divide our study into two cases: special and not special, according to the following
	
	\begin{defn}\label{exceptional case} 
		Let $\Gamma$ be an $n$-dimensional Bieberbach group with the cyclic holonomy group $G$ of square-free order. Let $D$ be the set of prime divisors of $|G|$ such that for each $p\in D$ the $n$-dimensional Bieberbach subgroups $\Gamma_p$ of $\Gamma$ with the cyclic holonomy group $C_p$ of order $p$, have corresponding $\Z C_p$-modules $M_p$ with all indecomposable summands of  $\Z$-rank $p-1$ except one trivial summand of $\Z$-rank $1$.  We sasy that $\Gamma$ is {\it special}, if $D\neq\emptyset$.
	\end{defn}
	
	Let $|G|=\delta=p_{1}p_{2}\cdots p_{k}$ be the decomposition of  $\delta$ into distinct primes. Denote by $\Q(\zeta_\delta)$ the cyclotomic field generated by a primitive $\delta$-th root of unity $\zeta_\delta$, by   $H(\Q(\zeta_\delta))$ its class group and by $\mathrm{Gal}(\zeta_\delta)$ its Galois group, the latter acts naturally on $H(\Q(\zeta_\delta))$. Using that
	
	\begin{equation}\label{characterization}
		\mathrm{Gal}(\zeta_\delta)\cong \mathrm{Gal}(\zeta_{p_{1}})\times\cdots \times \mathrm{Gal}(\zeta_{p_{k}})
	\end{equation} 
	(see \cite[Chapter 14, Corollary 27]{DF04}) and that  $\mathrm{Gal}(\zeta_{p_{i}})\cong C_{p_i-1}$ contains the unique subgroup $C_2$ or order 2,  for a set $D$  of prime divisors of $\delta$ we can define   a subgroup $\mathcal{H}_{D}$ of $\mathrm{Gal}(\zeta_\delta)$  that has $C_2$ instead of $\mathrm{Gal}(\zeta_{p})$ as a factor in the direct product in $\mathrm{Gal}(\zeta_{p_{1}})\times\cdots \times \mathrm{Gal}(\zeta_{p_{k}})$ for each $p_i\in D$.
	
	With these notations,  we obtain the following formula for $|\mathcal{C}(M)|$:
	
	\begin{theorem}\label{cardinality of the crystal class}
		Let $\Gamma$ be a  $n$-dimensional Bieberbach group with maximal abelian normal subgroup $M$  and cyclic holonomy group $G$ of square-free order $\delta$.  If $\Gamma$ is special, then  
		
		$$|\mathcal{C}(M)|=\left|\mathcal{H}_{D} \backslash \prod_{d\mid \delta}H(\Q(\zeta_d))\right|.$$  Otherwise, 
		$$|\mathcal{C}(M)|= \left|\mathrm{Gal}(\zeta_\delta)\backslash \prod_{d\mid \delta}H(\Q(\zeta_d))\right|. $$
	\end{theorem}
	
	Let $M$ be a $\Z C_p$-module (for a cyclic group $C_p=\langle x\rangle$ of prime order $p$). Then  $M$ can be written in the form $M=M_1\oplus M_2$, where $M_1$ is the largest direct summand of $M$ on which $C_p$ acts trivially.  For an element $n\in \mathcal{N}_{\mathrm{Aut}(M)}(C_p)$ denote by $\tilde n$ the its natural image in $\F_p^*=\F_p\setminus \{0\}$ under an identification $\mathrm{Aut}(C_p)\cong \F_p^*$.  
	
	Now let  "bar" denotes the reduction modulo $p$. 
	The normalizer $\mathcal{N}_{\mathrm{Aut}(M)}(C_p)$ acts in a natural way on the set of all fixed elements $M^{C_p}$ under the action of $C_p$ on $M$.  This induces the action of $\mathcal{N}_{\mathrm{Aut}(M)}(C_p)$  on $\bar{M}^{C_p}$, and hence $\mathcal{N}_{\mathrm{Aut}(M)}(C_p)$ acts on $\bar{M}^{C_p}/\Delta \cdot \bar{M}\cong \bar{M}_1$, where $\Delta=1+x+\cdots+x^{p-1}$ (see \cite[p. 168]{Cha86} for more details).

	Define a new action \textquotedblleft$\bullet$\textquotedblright \ of the normalizer $\mathcal{N}_{\mathrm{Aut}(M)}(C_p)$  on $\bar{M}_1$  by
	\begin{equation}\label{def of the action}
		n\bullet m := n\cdot \widetilde{n}  m, \ \ (m\in \bar{M}_1)
	\end{equation}
	where  \textquotedblleft$\cdot$\textquotedblright \ denotes the  action of $\mathcal{N}_{\mathrm{Aut}(M)}(C_p)$ on $\bar{M}_1$ described in the preceding paragraph.  Since for each prime $p$ dividing $|G|$ the normalizer   $\mathcal{N}_{\mathrm{Aut}(M)}(G)$ is a subgroup of $\mathcal{N}_{\mathrm{Aut}(M)}(C_p)$,  the normalizer  $\mathcal{N}_{\mathrm{Aut}(M)}(G)$ also acts on $M_1/pM_1$, and hence $\mathcal{N}_{\mathrm{Aut}(M)}(G)$ acts on $(M_1/pM_1)^{G/C_p}=(M_1/pM_1)^{G}$ for each prime $p$ dividing $|G|$.
	
	We  denote  $\bar M_1^{\ast}=\bar M_1-\{0\}$ and state the following
	
	\begin{theorem}\label{main Theorem}
		Let $\Gamma$ be an $n$-dimensional Bieberbach group with maximal abelian normal subgroup $M$ and cyclic holonomy group $G$ of square-free order $\delta$. Then,
		\begin{equation*}
			|\mathfrak{g}(\Gamma)|= \sum_{M\in T}\left(\prod_{p\mid \delta}|\mathcal{N}_{\mathrm{Aut}(M)}(G)\backslash (\bar{M}_{1,p}^{\ast})^{G}| \right),
		\end{equation*}
		where $T$ is a set of representatives for the isomorphism classes of $\Z G$-lattices in $\mathcal{C}(M)$ and $M_{1,p}$ is the largest direct summand of $M$ on which $C_p$ acts trivially.
	\end{theorem}
	
	\begin{cor}\label{cor1} $$|\mathfrak{g}(\Gamma)|\leq |H(\Q(\zeta_\delta))|^{a}\left( max\{|(\bar{M}_{1,p}^*)^{G}| \ : \ p\mid \delta\}\right) ^{b},$$ where $a$ is the number of divisors of $\delta$ and $b$ is the number of prime divisors of $\delta$. 
	\end{cor}
	\begin{cor}\label{cor2}
		If $\Gamma$ is a special Bieberbach group, then $$|\mathfrak{g}(\Gamma)|=\sum_{M\in T}\left( \prod_{p\in D}|\mathcal{N}_{\mathrm{Aut}(M)}(G)\backslash\bar{M}_{1,p}|\prod_{\substack{q\mid \delta \\ q \notin D}}|\mathcal{N}_{\mathrm{Aut}(M)}(G)\backslash (\bar{M}_{1,q}^{\ast})^{G}|\right) $$
	\end{cor}	
	
	We also deduce as a corollary the main result of \cite{Nery}.
	
	\begin{cor}\label{cor3}
		Let $\Gamma$ be an $n$-dimensional Bieberbach group with maximal abelian normal subgroup $M$ and cyclic holonomy group $G$ of square-free order. If  $|G|$ is a prime number, then $|\mathfrak{g}(\Gamma)|=|\mathcal{C}(M)|$.
	\end{cor}

	The outline of this paper is as follows. In Section \ref{Preliminaries}, we summarize without proofs the Oppenheim's classification of integral representations of a cyclic group of square-free order -- important result used in the proof of Theorem \ref{cardinality of the crystal class} -- and presents some preliminaries. Section \ref{Profinite genus} contains a special case of Theorem \ref{main Theorem} which generalizes the results of \cite{Nery} (see Theorem \ref{main Theorem2}) and  the proofs of our main results. We also give an example of use of the formula of Theorem \ref{main Theorem}.

	\subsection*{Notation}
	
	$\delta=$ square-free positive integer. \\
	$a\mid b=$ $a$ divides $b$ (for $a,b\in \Z$). \\
	$\zeta_d=$ primitive $d$-th root of unity (for $d\in \Z$). \\
	$\Phi_d(x)=$  $d$-th cyclotomic polynomial (for $d\in \Z$). \\
	$\Q(\zeta_d)=$ cyclotomic field generated by $\zeta_d$; $\Z[\zeta_d]=$ ring of integers of $\Q(\zeta_d)$ (for $d\in \Z$). \\
	$\mathrm{Gal}(\zeta_d)=$ Galois group of $\Q(\zeta_d)$ over $\Q$.\\
	$G=$ cyclic group of order $\delta$. \\
	$\widehat{\Gamma}=$ profinite completion of $\Gamma$ (see \cite{RZ00}). \\
	$H\backslash X=$ set of  orbits of the action of the group $H$ on a set $X$. \\
	$|X|=$ cardinality of the set $X$.
	
	\section{Preliminaries}\label{Preliminaries}
	
	\subsection{Modules over cyclic groups of square-free order}
	For the convenience of the reader we repeat the relevant results of the  Oppenheim's classification of integral representations of a cyclic group of square-free order \cite{Opp26} without proofs.
	
	Let $D_0$ denote the set of all $n\in \Z$ that divides $\delta$ such that there is only an even number of distinct primes in its decomposition. Similarly, let $D_1$ denote the set of all $n\in \Z$ that divides $\delta$ such that  there is only an odd number of distinct primes in its decomposition.   
	
	In $\Z[x]$ set 
	\begin{equation*}
		s_0(x)=\prod_{d\in D_0}\Phi_d(x), \ s_1(x)=\prod_{d\in D_1}\Phi_d(x)
	\end{equation*}
	and let $s_i=s_i(g), \ i=0,1$. 
	
	\begin{defn}  A  $\Z G$-lattice is  a $\Z G$-module which is finitely generated and free as a $\Z$-module.
	\end{defn}
	
	\begin{lem}[\cite{Opp26}, Lemma 3.4]\label{Lemma 3.4 of Opp26}
		Let $M$ be a  $\Z G$-lattice. Then, $$M_0=\{m\in M :s_0m=0 \}$$ and $M_1=M/M_0$ are $\Z G$-lattices.  
	\end{lem}
	
	Note that the module $M$ is an extension of $M_1$ by $M_0$, i.e., an exact sequence 
	\begin{equation*}
		1\rightarrow M_0\rightarrow M\rightarrow M_1\rightarrow 1
	\end{equation*} 
	of $\Z G$-modules. It is well known that there is a bijection between the set of equivalence classes of extensions of $M_1$ by $M_0$ and $\mathrm{Ext}_{\Z G}^{1}(M_1,M_0)$ (see \cite{Rot09} for details). Thus, to characterize a  $\Z G$-lattice $M$ we must determine not only the structure of $M_0$ and $M_1$ but also the group $\mathrm{Ext}_{\Z G}^{1}(M_1,M_0)$.
	
	\begin{prop}[\cite{Opp26}, p. 17]
		Let $M_0$ and $M_1$ be as in Lemma \ref{Lemma 3.4 of Opp26}. Then, $$M_i\cong \bigoplus_{d\in D_i}\mathcal{M}_d$$ where $\mathcal{M}_d:= t_dM_i$ with $t_d:= s_i/\Phi_d(g), \ i=0,1$.
	\end{prop}
	
	Note that $s_0M_0=0=s_1M_1$. Thus, $\Phi_d(g)\mathcal{M}_d=0$ for all $d \mid \delta$. Hence, $\mathcal{M}_d$ is a $\Z G/\Phi_d(g)\Z G$-module.
	
	Let $\zeta_d$ be a primitive $d$-th root of unity.  We have a ring isomorphism 
	\begin{equation*}
		\frac{\Z G}{\Phi_d(g) \Z G} \cong \Z[\zeta_d]
	\end{equation*}
	given by $g\mapsto \zeta_d$. Thus, we may then turn $\mathcal{M}_d$ into a $\Z[\zeta_d]$-module by setting  $ \zeta_d\cdotp m := gm, \ m\in \mathcal{M}_d$. Since $\Z[\zeta_d]$ is a Dedekind domain, it follows that we may write
	\begin{equation*}
		\mathcal{M}_d \cong I_{1,d}\oplus I_{2,d}\oplus\cdots\oplus I_{r(d),d}
	\end{equation*}
	where the $\{I_{j,d}\}$ are ideals in $\Z[\zeta_d]$ (see \cite[Theorem 4.13]{CR81}). Moreover, the isomorphism invariants of $\mathcal{M}_d$ are its rank $r(d)$ and the ideal class of the product $I_{1,d}I_{2,d}\cdots I_{r(d),d}$.
	
	\begin{prop}[\cite{Opp26}, Theorem 4.1 and Corollary 4.8]\label{Th 4.1 and Cor 4.8 of Opp26}
		Let $M_0$ and $M_1$ be as in Lemma \ref{Lemma 3.4 of Opp26}. Then, 
		\begin{equation*}
			\mathrm{Ext}_{\Z G}^{1}(M_1,M_0) \cong\bigoplus_{(s,t)\in D_1\times D_0} \Lambda(s,t)
		\end{equation*}
		where $\Lambda(s,t)$ is the $\Z[\zeta_s]/\Phi_t(\zeta_s)\Z[\zeta_s]$-module of $r(s)\times r(t)$ matrices with entries in $\Z[\zeta_s]/\Phi_t(\zeta_s)\Z[\zeta_s]$.
	\end{prop}
	Let $L(s,t)$ denote $\Z[\zeta_s]/\Phi_t(\zeta_s)\Z[\zeta_s]$. Suppose now that $M$ is an extension of $M_1$ by $M_0$ corresponding to an element $\lambda \in \mathrm{Ext}_{\Z G}^{1}(M_1,M_0)$. According to Proposition \ref{Th 4.1 and Cor 4.8 of Opp26} we can write 
	\begin{equation*}
		\lambda=(\lambda(s_1,t_1),\cdots,\lambda(s_l,t_l))
	\end{equation*}
	where $(s_i,t_i)\in D_1\times D_0$ and $\lambda(s_i,t_i)$ is an $r(s_i)\times r(t_i)$ matrix with entries in $L(s_i,t_i), \ i=1,\cdots,l$.
	
	\begin{lem}[\cite{Opp26}, Lemma 4.4]\label{lem 4.4 of Opp26}
		Let $s\mid \delta, \ t\mid \delta$ and $t>s$. Then, $L(s,t)$  is a trivial ring unless there is a prime $p$ such that $t=ps$. If $t=ps$, then $L(s,t)=\mathrm{F}_{p,1}\oplus\cdots\oplus \mathrm{F}_{p,v}$, where $\mathrm{F}_{p,i}$ is a field of characteristic $p$, $i=1,\cdots,v$.
	\end{lem}
	
	Thus, the entries $a_{ij} \in L(s,t)$ of the matrix $\lambda(s,t)$ can be written as $$a_{ij}=(\alpha_{ij}^{1},\cdots,\alpha_{ij}^{v})$$ where $\alpha_{ij}^{k}\in \mathrm{F}_{p,k}$.  Then to $\lambda(s,t)$ corresponds the $v$-tuple of matrices $((\alpha_{ij}^{1}),\cdots,(\alpha_{ij}^{v}))$. Set $\rho_{k}(\lambda(s,t))=\mathrm{rank}(\alpha_{ij}^{k})$.
	
	We need a little more notation to state the classification of the  $\Z G$-lattices. Let 
	\begin{equation*}
		D^{\ast}= \{(s,t) : s\mid\delta, t\mid\delta \ \text{and} \ s/t \ \text{is a prime power} \} 
	\end{equation*}
	and let $D_{1}^{\ast}=\{(s,t) : (s,t)\in D^{\ast}, s\in D_1 \}$.
	
	\begin{prop}[\cite{Opp26}, Theorem 4.13]\label{Theorem 4.13 of Opp26}
		Let $M$ be a $\Z G$-lattice. A full set of isomorphism invariants of $M$ consists of:
		\begin{enumerate}
			\item[(i)] The $\Z G$-rank of $\mathcal{M}_d$, $r(d)$, for each $\mathcal{M}_d$ and  $d\mid \delta$.
			\item[(ii)] The ideal class of product $I_{1,d}I_{2,d}\cdots I_{r(d),d}$ associated with $\mathcal{M}_d$, for each $d\mid \delta$.
			\item[(iii)] $\{\rho_k(\lambda(s,t)) :  (\lambda(s,t))\in \mathrm{Ext}_{\Z G}^{1}(M_1,M_0), (s,t)\in D_{1}^{\ast}, k=1,\cdots,v  \}$.
		\end{enumerate}
	\end{prop}
	In what follows, $r(d,M)$ and $\lambda(s,t,M)$ denote $r(d)$ and $\lambda(s,t)$ on the module $M$, respectively. Moreover, to shorten notation, let $[I_{\mathcal{M}_d}]$ denote the ideal class of product  $I_{1,d}I_{2,d}\cdots I_{r(d),d}$ associated with $\mathcal{M}_d$, for each $d\mid \delta$.
	
	By \cite[Theorem 7.3]{Opp26} and \cite[Proposition 31.15]{CR81} we have the following profinite version of Proposition \ref{Theorem 4.13 of Opp26}.
	
	\begin{prop}\label{profinite version of Th 4.13 of Opp26}
		Let $M$ and $N$ be $\Z G$-lattices. Then, $\widehat M\cong \widehat N$ as $\widehat{\Z}G$-modules if and only if
		\begin{enumerate}
			\item[(i)] $r(d,M)=r(d,N)$ for each $d\mid \delta$.
			\item[(ii)] $\rho_k(\lambda(s,t,M))=\rho_k(\lambda(s,t,N))$ for each $k=1,\cdots,v$ and $(s,t)\in D_{1}^{\ast}$.
		\end{enumerate}
	\end{prop}
	
	The next result will be useful to us, later on.
	
	\begin{prop}[\cite{Wie84}, Proposition 5.1]\label{Prop 5.1 of Wie84}
		Let $C_n$ be a cyclic group of order $n$ ($n$ square-free) and let $M,N,$ and $L$ be $\Z C_n$-lattices. Then, $M\oplus L\cong N\oplus L$ implies $M\cong N$ if and only if $n=6,10,14$, or $n$ is prime.
	\end{prop}
	The next result tells us that the direct-sum cancellation holds for $\widehat{\Z} G$-modules.
	\begin{prop}\label{thm 7.4 of Opp26}
		Let $M_1,M_2, N_1$ and $N_2$ be $\Z G$-lattices. Then, 
		$$\widehat{M}_1\oplus\widehat{M}_2\cong\widehat{N}_1\oplus\widehat{N}_2$$ if and only if $$r(d,M_1)+r(d,M_2)=r(d,N_1)+r(d,N_2)$$ and $$\rho_k(\lambda(s,t,M_1))+\rho_k(\lambda(s,t,M_2))=\rho_k(\lambda(s,t,N_1))+\rho_k(\lambda(s,t,N_2))$$ for each $d\mid\delta$ and $k=1,2,\cdots, v$.
	\end{prop}
	\begin{proof}
		The result follows from \cite[Theorem 7.4]{Opp26} and \cite[Proposition 31.15]{CR81}.
	\end{proof}
	
	\subsection{Galois groups acting on ideal class groups}
	
	Let $\zeta_m$ be a primitive $m$-th root of unity. It is well known that the Galois group $\mathrm{Gal}(\zeta_m)$ of the cyclotomic field $\Q(\zeta_m)$ is isomorphic to $(\Z/m\Z)^{\times}$ (the multiplicative group of units of the ring $\Z/m\Z$).
	
	\begin{prop}[\cite{DF04}, Chapter 14, Corollary 27]\label{Cor 27 of DF04}
		Let $m=p_{1}^{a_1}p_{2}^{a_2}\cdots p_{k}^{a_k}$ be the decomposition of the positive integer $m$ into distinct prime powers. Then, 
		\begin{equation*}
			\mathrm{Gal}(\zeta_m)\cong \mathrm{Gal}(\zeta_{p_{1}^{a_1}})\times\cdots \times \mathrm{Gal}(\zeta_{p_{k}^{a_k}}).
		\end{equation*} 
	\end{prop}
	
	The next lemma shows that there is an action of the Galois group $\mathrm{Gal}(\zeta_m)$ on the ideal class group $H(\Q(\zeta_m))$.
	
	\begin{lem}[\cite{Cha86}, Chapter IV, Exercise 6.2]\label{Exercise 6.2}
		Let $A$ and $B$ be ideals of $\Z[\zeta]$. If $A$ and $B$ are in the same ideal class, then $\sigma(A)$ and $\sigma(B)$ are in the same ideal class for any $\sigma\in\mathrm{Gal}(\zeta_m)$.
	\end{lem}
	
	\begin{remark}
		\begin{enumerate}
			\item[(i)] Given any positive square-free integer $\delta$   suppose that $d\mid \delta$. Then, there is $u\in \Z$ such that $\delta=du$. Since $\zeta^{u}$ is a primitive $d$-th  root of unity, the field $\Q(\zeta_d)$ is a subfield of $\Q(\zeta_{\delta})$. Thus, for any $\sigma$ in $\mathrm{Gal}(\zeta_{\delta})$ we have $\sigma|_{\Q(\zeta_d)}\in\mathrm{Gal}(\zeta_d) $. Therefore, $\mathrm{Gal}(\zeta_{\delta})$ acts via automorphisms on $H(\Q(\zeta_d))$ for each $d\mid \delta$.
			\item[(ii)] From (i) it follows that $\mathrm{Gal}(\zeta_{\delta})$ acts via automorphisms on the direct product $\prod_{d\mid\delta}H(\Q(\zeta_d))$. We let
			$\mathrm{Gal}(\zeta_\delta)\backslash \prod_{d\mid \delta}H(\Q(\zeta_d))$ be the orbit set.
		\end{enumerate}		
	\end{remark}
	
	\begin{defn}
		Let $M$ and $N$ be $\Z G$-modules. A semi-linear homomorphism from $M$ to $N$  is a pair $(f,\varphi)$  where $f:M\rightarrow N$ is an abelian group homomorphism and $\varphi$ is  an automorphism of $G$ such that $$f(x\cdotp m)=\varphi(x)\cdotp f(m)$$ for $x\in G$ and $m\in M$.
	\end{defn}
	Let $\varphi$ be an automorphism of a finite group $G$ and let $M$ be a $\Z G$-module. Let $(M)^{\varphi}$ denote the $\Z G$-module $M$ given by the action $g\star m:=\varphi(g)\cdotp m$ for $g\in G$ and $m\in M$.
	
	\begin{prop}\label{Invariants_Semi-Linear Iso}
		Let $M$ and $N$ be $\Z G$-lattices. Then, $M$ will be semi-linearly isomorphic to $N$ if and only if
		\begin{enumerate}
			\item[(i)] $r(d,M)=r(d,N)$ for each $d\mid \delta$.
			\item[(ii)] $\rho_k(\lambda(s,t,M))=\rho_k(\lambda(s,t,N))$ for each $k=1,\cdots,v$ and $(s,t)\in D_{1}^{\ast}$.
			\item[(iii)] $\sigma\cdotp [I_{\mathcal{M}_d}]=[J_{\mathcal{N}_d}]$ for each $d\mid \delta$ and for some $\sigma\in \mathrm{Gal}(\zeta_{\delta})$.
		\end{enumerate}
	\end{prop}
	\begin{proof}
		Let $M$ and $N$ be $\Z G$-lattices. Note that $M$ and $N$ are semi-linearly isomorphic if and only if $M\cong (N)^{\varphi}$ as $\Z G$-modules for some $\varphi\in \mathrm{Aut}(G)$. As $G$ is a cyclic group of square-free order $\delta$, by Proposition  \ref{Theorem 4.13 of Opp26} we have
		\begin{enumerate}
			\item[(i)] $r(d,M)=r(d,N)$ for each $d\mid\delta$.
			\item[(ii)] $\rho_k(\lambda(s,t,M))=\rho_k(\lambda(s,t,N))$ for each $k=1,\cdots,v$ and $(s,t)\in D_{1}^\ast$.
			\item[(iii)] $I_{\mathcal{M}_d}\cong (J_{\mathcal{N}_d})^{\varphi}$ as $\Z[\zeta_d]$-modules for each $d\mid \delta$ and for some $\varphi\in \mathrm{Aut}(G)$ (see \cite[Lemma 22.2]{CR62} ).
		\end{enumerate}
		For each $d\mid \delta$, we will denote by $C_d$ the subgroup of $G$ of order $d$. Recall that  $\mathrm{Gal}(\zeta_{\delta})\cong (\Z/\delta\Z)^{\times}\cong \mathrm{Aut}(G)$. So to finish the proof it suffices to show that $(J_{\mathcal{N}_d})^{\varphi}\cong \varphi^{-1}\cdotp J_{\mathcal{N}_d}$ as $\Z C_d$-modules. Suppose $\varphi(c)=c^{l}$ where $c\in C_d$ and $l$ is a positive integer such that $(l,\delta)=1$. Consider the map $\eta: (J_{\mathcal{N}_d})^{\varphi}\rightarrow \varphi^{-1}\cdotp J_{\mathcal{N}_d}$ defined by $u\mapsto \varphi^{-1}(u)$. It is clear that $\eta$ is a  group isomorphism, so we have to show that $\eta$ is a $C_d$-homomorphism.  Note that 
		\begin{eqnarray*}
			\eta(c\star u) &=& \eta(\varphi(c)\cdotp u) \\
			&=&  \varphi^{-1}(c^{l}\cdotp u) \\
			&=& \varphi^{-1}(\zeta_{d}^{l} u) \\
			&=&\varphi^{-1}(\zeta_{d}^{l})\varphi^{-1}(u) \\
			&=& \zeta_{d}\varphi^{-1}(u) \\
			&=& c\cdotp \eta(u)
		\end{eqnarray*}
		where $c\in C_d$.  This finishes the proof of Proposition \ref{Invariants_Semi-Linear Iso}.
	\end{proof}
	
	\subsection{Bieberbach groups and Cohomology of Groups}
	
	If $\Gamma$ is any Bieberbach group, then $\Gamma$ satisfies an exact sequence
	\begin{equation}\label{exact sequence}
		1\rightarrow M\rightarrow \Gamma\rightarrow H\rightarrow 1,
	\end{equation}
	where $M$ is a maximal abelian subgroup (subgroup of all translations) of $\Gamma$ and $H$ is a finite group. 
	
	\begin{prop}[\cite{Cha86}, Chapter I, Proposition 4.1]\label{Cha, Prop 4.1}
		Let $\Gamma$ be an $n$-dimensional Bieberbach group. Then, the translation subgroup $M$ is the unique normal, maximal abelian subgroup of $\Gamma$.
	\end{prop}

	Note that the exact sequence \eqref{exact sequence} gives a natural structure of $\Z H$-module on $M$. Since $M$ is maximal abelian,  it follows that $M$ is a faithful $\Z H$-module. Recall that a crystal class $(H,M)$ is the set all  $n$-dimensional Bieberbach groups $\Gamma$ that appears as an extension of $H$ by $M$ and that two crystal classes $(H,M)$ and $(H',M')$ are arithmetically equivalent if there exist isomorphisms $\varphi:H\rightarrow H'$ and $f:M\rightarrow M'$ such that $f(h\cdot m)=\varphi(h)\cdot f(m)$ for all $h\in H$ and $m\in M$, i.e., if $H$ and $H'$ are conjugate subgroups of $\mathrm{GL}(n,\Z)$. The resulting equivalence classes are the arithmetic crystal classes.
	
	\begin{lem}\label{Iso B gropus implies iso sl iso} 
		Isomorphic Bieberbach groups determine the same arithmetic crystal class.
	\end{lem}
	\begin{proof}
		Let $\Gamma_1$ and $\Gamma_2$ be $n$-dimensional Bieberbach groups with the holonomy groups $H_1, H_2$  and maximal abelian normal subgroups $M_1,M_2$, respectively. Suppose that  $F:\Gamma_1\rightarrow \Gamma_2$ is an isomorphism. By Proposition \ref{Cha, Prop 4.1}, $F(M_1)=M_2$. Hence, $F$ induces an isomorphism $\varphi:H_1\rightarrow H_2$. Let $\pi_1:\Gamma_1\rightarrow H_1$ and $\pi_2:\Gamma_2\rightarrow H_2$ be the natural projections. Let $\gamma_1\in \Gamma_1$ and suppose $\pi_1(\gamma_1)=h_1\in H_1$. By the definition of $\varphi$, we have $\pi_2(F(\gamma_1))=\varphi(h_1)$. Hence, if $f$ denotes the restriction of $F$ to $M$, then	
		\begin{equation*}
			f(h_1\cdotp m)=F(\gamma_1 m \gamma_{1}^{-1})=F(\gamma_1)F(m)F(\gamma_1)^{-1}=\varphi(h_1)\cdotp f(m),
		\end{equation*} 
		where $m\in M_1$. Therefore, $\Gamma_1$ and $\Gamma_2$ determine the same arithmetic crystal class.
	\end{proof}
	
	\begin{lem}\label{Independence of the CC}
		Let $(H,M)$ and $(H',M')$ be  belong to the same arithmetic crystal class. Then, for each Bieberbach group $\Gamma$ in $(H,M)$ there is an isomorphic group $\Gamma'$ in $(H',M')$.
	\end{lem}
	\begin{proof}
		Let $\Gamma$ be a Bieberbach group of the crystal class $(H,M)$. Let $\xi: H\times H\rightarrow M$ be a $2$-cocycle corresponding to the extension
		\begin{equation*}
			1\rightarrow M\rightarrow \Gamma \rightarrow H\rightarrow 1,
		\end{equation*}
		and consider $\Gamma$ to be $M\times H$ with the multiplication
		\begin{equation*}
			(m_1,h_1)(m_2,h_2)=(m_1+h_1m_2+\xi(h_1,h_2),h_1h_2),
		\end{equation*}
		where $m_1,m_2\in M$ and $h_1,h_2\in H$. Since $(H,M)$ and $(H',M')$ are in the same arithmetic crystal class, there exist isomorphisms $\varphi:H\rightarrow H'$ and $f:M\rightarrow M'$ such that  $$f(h\cdot m)=\varphi(h)\cdot f(m)$$ for all $h\in H$ and $m\in M$. Now define a map $F:M\times H\rightarrow M'\times H'$ by $$F(m,h)=(f(m),\varphi(h))$$ for $m\in M$ and $h\in H$. Clearly, $F$ is a bijective and 
		\begin{eqnarray*}
			F[(m_1,h_1)(m_2,h_2)]&=& F(m_1+h_1m_2+\xi(h_1,h_2),h_1h_2) \\
			&=& (f(m_1)+\varphi(h_1)f(m_2)+f(\xi(h_1,h_2)),\varphi(h_1)\varphi(h_2)) \\
			&=& (f(m_1)+\varphi(h_1)f(m_2)+f(\xi( f^{-1}fh_1 f^{-1}f,f^{-1}fh_2f^{-1}f)),\varphi(h_1)\varphi(h_2)) \\
			&=& (f(m_1)+\varphi(h_1)f(m_2)+f(\xi( f^{-1}\varphi(h_1)f,f^{-1}\varphi(h_2)f)),\varphi(h_1)\varphi(h_2)) \\
			&=& (f(m_1)+\varphi(h_1)f(m_2)+\xi '(\varphi(h_1),\varphi(h_2)),\varphi(h_1)\varphi(h_2)) \\
			&=& (f(m_1),\varphi(h_1))(f(m_2),\varphi(h_2)) 
		\end{eqnarray*}
		where $\xi'(\varphi(h_1),\varphi(h_2))=f( \xi(f^{-1}\varphi(h_1)f,f^{-1}\varphi(h_2)f))$ is a $2$-cocycle from $H'\times H'$ to $M'$. Therefore, $F(\Gamma)=\Gamma'$ is a Bieberbach group of the crystal class $(H',M')$.
	\end{proof}
	
	\begin{remark}\label{remark-Independence of the CC}
		It follows from Lemmas \ref{Iso B gropus implies iso sl iso} and \ref{Independence of the CC} that to find all possible Bieberbach groups (up to isomorphism) of an arithmetic crystal class it is sufficient to find all possible Bieberbach groups of only one representative of the class.
	\end{remark}
	
	Recall that the groups $\Gamma$ that satisfies an exact sequence \eqref{exact sequence}  can be classified  by elements of the second cohomology group  $H^{2}(H,M)$ of $H$ with coefficients in $M$. It is known that if $H$ is a finite group and $M$ is a finitely generated $\Z H$-module, then $H^{2}(H,M)$ is a finite group (see \cite[Corollary 9.41]{Rot09}). 
	
	\begin{prop}[\cite{Cha86}, Chapter III, Theorem 2.1]\label{Th 2.1, CIII, Cha86}
		Let $\Gamma$ be the extension corresponding to $\alpha \in H^{2}(H,M)$. Then, $\Gamma$ is torsion-free if and only if, for any cyclic subgroup $C_p$ of $H$ of prime order $p$, the image of the restriction homomorphism $res^{2}(\alpha)\in H^{2}(C_p,M)$ is not zero.
	\end{prop}
	
	\begin{defn}
		An element $\alpha \in H^{2}(H,M)$ is special, if it defines a Bieberbach group.
	\end{defn}
	
	Let $X(H,M)$ denote the subset of $H^{2}(H,M)$ formed by all the special elements.
	
	We need a little more notation to characterize all possible Bieberbach groups in the crystal class. The normalizer $\mathcal{N}_{\mathrm{Aut}(M)}(H)$ of $H$ in $\mathrm{Aut}(M)$ acts on $H^{2}(H,M)$ by the formula 
	\begin{equation}\label{def of the action2}
		[\phi\circledast c](h_1,h_2)=\phi(c(\phi^{-1}h_1\phi,\phi^{-1}h_2\phi)),
	\end{equation}
	where $c: H\times H\rightarrow M$ is a $2$-cocycle,  $\phi\in \mathcal{N}_{\mathrm{Aut}(M)}(H)$, and $h_1,h_2\in H$ (see \cite[p. 168]{Cha86}). 
	
	\begin{remark}
		Let $M$ be a $\Z C_p$-module for a cyclic group $C_p=\langle x\rangle$ of prime order $p$. Then, $M$ can be written in the form $M=M_1\oplus M_2$, where $M_1$ is the largest direct summand of $M$ on which $C_p$ acts trivially. Thus, $H^2(C_p,M)\cong \bar{M}_1$ (see \cite[Chap. IV, Theorem 5.1]{Cha86}). Therefore, we can define an action of $\mathcal{N}_{\mathrm{Aut}(M)}(C_p)$ on $H^{2}(C_p,M)$ by the formula \eqref{def of the action}. We claim that this action is equal to the action \textacutedbl$\circledast$\textgravedbl defined in \eqref{def of the action2}. Indeed, let $\phi\in \mathcal{N}_{\mathrm{Aut}(M)}(C_p)$ and define $\widetilde{\phi}\in \mathrm{Aut}(C_p)$ by $\widetilde{\phi}(x)=\phi^{-1}x\phi$. Note that if $c:C_p\times C_p\rightarrow M$ is a $2$-cocycle, then $$[\widetilde{\phi}_{\ast}(c)](x,y)=c(\widetilde{\phi}(x),\widetilde{\phi}(y))$$ and $$[\phi_{\ast}(c)](x,y)=\phi(c(x,y)),$$ for $x,y\in C_p$, define homomorphisms $\widetilde{\phi}_{\ast}:H^2(C_p,M)\rightarrow H^2(C_p,M)$ and $\phi_{\ast}:H^2(C_p,M)\rightarrow H^2(C_p,M)$, respectively. Thus, the action \textacutedbl$\circledast$\textgravedbl defined in \eqref{def of the action2} can be rewritten as 
		\begin{equation}\label{def1}
			\phi\circledast \alpha=\phi_{\ast}(\widetilde{\phi}_{\ast}(\alpha)),
		\end{equation}
		where $\alpha\in H^2(C_p,M)$. On the other hand, since $H^2(C_p,M)\cong \bar{M}_1$, the action  \textacutedbl$\bullet$\textgravedbl defined in \eqref{def of the action} can be rewritten as
		\begin{equation}\label{def2}
			\phi\bullet m= \phi\cdot \widetilde{\phi}m=\phi_{\ast}(\widetilde{\phi}_{\ast}(m)),
		\end{equation} 
		where $m\in \bar{M}_1$. Therefore, from \eqref{def1} and \eqref{def2} we see that the action \textacutedbl$\circledast$\textgravedbl defined in \eqref{def of the action2} is equal to the action \textacutedbl$\bullet$\textgravedbl defined in \eqref{def of the action} in this case.
	\end{remark}
	
	\begin{prop}[\cite{Hil86}, Theorem 5.2]\label{Theorem of Mathematical Crystallography}
		There exists a one-to-one correspondence between the isomorphism classes of Bieberbach groups in the  crystal class $(H,M)$  and the orbits of the action of   $\mathcal{N}_{\mathrm{Aut}(M)}(H)$ on $X(H,M)$. 
	\end{prop}
	
	\begin{lem}\label{acts transit}
		Let $C_p$ be a cyclic group of prime order $p$. Then, $\mathcal{N}_{\mathrm{Aut}(M)}(C_p)$ acts transitively on $H^{2}(C_p,M)^{\ast}$.
	\end{lem}
	
	To prove this, we will need the following definition:

	\begin{defn}\label{exceptional module}
		Let $C_p$ be a cyclic group of prime order $p$. A finitely generated  $\Z C_p$-module $M$ is exceptional if all indecomposable summands of $M$ have $\Z$-rank $p-1$ except one trivial summand of $\Z$-rank $1$.	
	\end{defn}
	
	\begin{proof}[Proof of Lemma \ref{acts transit}]
		Suppose that $M$ is an exceptional $\Z C_p$-module, i.e.,  $M=\bigoplus_{i=1}^{r}A_i\oplus \Z$ where $A_i$ are ideals of $\Z[\zeta_p]$. For each $\varphi\in \mathrm{Gal}(\zeta_p)$, consider the map $\phi: \bigoplus_{i=1}^{r}A_i\oplus \Z\rightarrow \bigoplus_{i=1}^{r}\varphi (A_i)\oplus \Z$ defined by $$a_1\oplus\cdots\oplus a_r\oplus u\mapsto \varphi(a_1)\oplus\cdots\oplus \varphi(a_r)\oplus u,$$ for $a_i\in A_i$ and $u\in \Z$. Recall that $\mathrm{Gal}(\zeta_p)\cong\mathrm{Aut}(C_p)$; every  $\varphi\in \mathrm{Aut}(C_p)$ is of the form $\varphi(x)=x^{k}$  where $x$ is a generator of $C_p$ and $k$ is some integer between $0$ and $p$. Thus, 
		\begin{equation*}
			\varphi(x\cdot a_i)=\varphi(\zeta_pa_i)=\varphi(\zeta_p)\varphi(a_i)=\varphi(x)\cdot\varphi(a_i), 
		\end{equation*}
		for $i=1,\cdots,r$. Since $\Z$ is a trivial $\Z C_p$-module, we have $\phi(x\cdot m)=\varphi(x)\cdot \phi(m)$ for each $m\in M$.  Hence, $\phi \in \mathcal{N}_{\mathrm{Aut}(M)}(C_p)$. As $H^{2}(C_p,M)=H^{2}(C_p,\Z)\cong \bar{\Z}$, we have $$\phi \bullet m= \widetilde{\phi}m = km,$$ where $m \in \bar{\Z}$ and $\widetilde{\phi}=k$. Therefore, by construction we have $\mathcal{N}_{\mathrm{Aut}(M)}(C_p)$ acts transitively on $H^{2}(C_p,M)^{\ast}$ in this case.
		
		Now, if $M$ is a non-exceptional $\Z C_p$-module, then $\mathcal{N}_{\mathrm{Aut}(M)}(C_p)$ also acts transitively on $H^{2}(C_p,M)^{\ast}$ by \cite[Theorem 6.1, p. 140]{Cha86}.
	\end{proof}
	
	Let $p$ be a prime number. Let $H_p$ denote the $p$-primary component of the group $H$.
	
	\begin{prop}[\cite{Brown82}, p. 84]\label{p-primary component}
		Let $H$ be a finite group and $P$ a normal $p$-Sylow subgroup of $H$. Then,
		\begin{enumerate}
			\item[(i)] $H^n(H,M)= \bigoplus_{p\mid |H|}H^{n}(H,M)_{p}$ for all $n>0$.
			\item[(ii)]  $H^{n}(H,M)_{p}\cong H^{n}(P,M)^{G/P}$ for all $n>0$.
		\end{enumerate}
	\end{prop}

	In particular, if $G$ is a cyclic group of square-free order, then  $X(G,M)=\prod_{p\mid|G|} (\bar{M}_{1,p}^{\ast})^{G/C_p}=\prod_{p\mid|G|} (\bar{M}_{1,p}^{\ast})^{G}$ where $M_{1,p}$ is the largest direct summand of $M$ on which $C_p$ acts trivially (see Proposition \ref{Th 2.1, CIII, Cha86}).
	
	\begin{lem}\label{orbit of N_{GL}(G)}
		Let $M$ be a $\Z G$-lattice for a cyclic group $G$ of square-free order $\delta$. Then, 
		\begin{equation}\label{formula1}
			|\mathcal{N}_{\mathrm{Aut}(M)}(G)\backslash X(G,M)|=\prod_{p\mid\delta}|\mathcal{N}_{\mathrm{Aut}(M)}(G)\backslash (\bar{M}_{1,p}^{\ast})^{G}|,
		\end{equation}
		where  $M_{1,p}$ is the largest direct summand of $M$ on which $C_p$ acts trivially. 
	\end{lem}
	\begin{proof}
		Since $\mathcal{N}_{\mathrm{Aut}(M)}(G)$ is a subgroup of $\mathcal{N}_{\mathrm{Aut}(M)}(C_p)$ for each $p\mid \delta$, the result follows immediately from Proposition \ref{p-primary component}.
	\end{proof}
	
	\section{Profinite genus}\label{Profinite genus}

	\subsection{Bieberbach groups with the cyclic holonomy group of square-free order}
	In this section we shall determine a full set of isomorphism invariants for the profinite completion of an $n$-dimensional Bieberbach group with the cyclic holonomy group of square-free order and also for an arithmetic crystal class.
	
	\begin{lem}\label{characterization-Bieberbach group with n(G)=1}
		Let $\Gamma$ be an $n$-dimensional Bieberbach group with cyclic holonomy group $G$ of square-free order $\delta$ and  $M$ its maximal abelian normal subgroup. Then, $M=M_{n-1}\oplus \Z$ admits a $\Z G$-decomposition, where $\Z$ is trivial module generated by the $\delta$-th power of some element $c$ of $\Gamma$ and $\Gamma=M_{n-1}\rtimes C$ with $C=\langle c\rangle$.
	\end{lem}
	\begin{proof}
		By \cite[Theorem 2]{Szc97} we have the following exact sequence
		\begin{equation*}
			1\rightarrow M_{n-1}\rightarrow \Gamma\rightarrow \Z \rightarrow 1.
		\end{equation*} 
		Since $\Z$ is free, this sequence splits as a semidirect product $\Gamma=M_{n-1}\rtimes C$.
	\end{proof}
	
	\begin{lem}\label{Invariants_crystal class}
		Let $\Gamma_1$ and $\Gamma_2$ be $n$-dimensional Bieberbach groups with the cyclic holonomy group $G$ of square-free order $\delta$ and maximal abelian normal subgroups $M$ and $N$, respectively. Then, $\Gamma_1$ and $\Gamma_2$ are in the same arithmetic crystal class if and only if
		\begin{enumerate}
			\item[(i)] $r(d,M)=r(d,N)$ for each $d\mid \delta$.
			\item[(ii)] $\rho_k(\lambda(s,t,M))=\rho_k(\lambda(s,t,N))$ for each $k=1,\cdots,v$ and $(s,t)\in D_{1}^{\ast}$.
			\item[(iii)] $\sigma\cdotp [I_{\mathcal{M}_d}]=[J_{\mathcal{N}_d}]$ for each $d\mid \delta$ and for some $\sigma\in \mathrm{Gal}(\zeta_{\delta})$.
		\end{enumerate}
	\end{lem}
	\begin{proof}
		This is a consequence of Proposition \ref{Invariants_Semi-Linear Iso}.
	\end{proof}

	\begin{prop}\label{abstract invariants}
		Let $\Gamma_1$ and $\Gamma_2$ be $n$-dimensional Bieberbach groups with the cyclic holonomy group $G$ of order $\delta=6,10,14$, or $\delta$ is prime and maximal abelian normal subgroups $M$ and $N$, respectively. Then, $\Gamma_1\cong\Gamma_2$ if and only if
		\begin{enumerate}
			\item[(i)] $r(d,M)=r(d,N)$ for each $d\mid \delta$.
			\item[(ii)] $\rho_k(\lambda(s,t,M))=\rho_k(\lambda(s,t,N))$ for each $k=1,\cdots,v$ and $(s,t)\in D_{1}^{\ast}$.
			\item[(iii)] $\sigma\cdotp [I_{\mathcal{M}_d}]=[J_{\mathcal{N}_d}]$ for each $d\mid \delta$ and for some $\sigma\in \mathrm{Gal}(\zeta_{\delta})$.
		\end{enumerate}
	\end{prop}
	\begin{proof}
		If $\Gamma_1\cong\Gamma_2$, then $M$ is semi-linearly isomorphic to $N$ by Lemma \ref{Iso B gropus implies iso sl iso}. Therefore, it follows from Proposition \ref{Invariants_Semi-Linear Iso} that the conditions (i)--(iii) are fulfilled.
		
		Conversely, let $M$ and $N$ be $\Z G$-lattices  satisfying the conditions (i)--(iii). By Lemma \ref{characterization-Bieberbach group with n(G)=1}, $M=M_{n-1}\oplus \Z$ and $N=N_{n-1}\oplus \Z'$ such that $\Gamma_1=M_{n-1}\rtimes C_1$ and $\Gamma_2=N_{n-1}\rtimes C_2$ where $C_1, C_2$ contain $\Z$ and $\Z'$ as subgroups of index $\delta$ and acts on $M_{n-1}, N_{n-1}$ as $G$. Since the conditions (i)--(iii) are satisfied for $M_{n-1}\oplus \Z$ and $N_{n-1}\oplus \Z'$, it follows from Proposition \ref{Invariants_Semi-Linear Iso} that $M_{n-1}\oplus \Z$ is semi-linearly isomorphic to $N_{n-1}\oplus \Z'$. Hence, there is a semi-linearly isomorphism $(f,\varphi)$ from $M_{n-1}$ to $N_{n-1}$ by Proposition \ref{Prop 5.1 of Wie84}. Let $\Theta:C_1\rightarrow C_2$ be the isomorphism induced by $\varphi: G\rightarrow G$. Now, consider the map $F: M_{n-1}\rtimes C_1\rightarrow N_{n-1}\rtimes C_2$ given by $$F(m,c)=(f(m),\Theta(c)),$$ for $m\in M_{n-1}$ and $c\in C_1$.  We have $F$ is a group homomorphism, because  $(f,\varphi)$ is a semi-linearly homomorphism from $M_{n-1}$ to $N_{n-1}$. Since $F$ is clearly bijective, it is an isomorphism, as we wanted.	
	\end{proof}
	
	\begin{lem}[\cite{Nery}, Lemma 3.2]\label{diagram}
		Let $\Gamma_1$ and $\Gamma_2$ be $n$-dimensional Bieberbach groups with  maximal abelian normal subgroups $M_1$ and $M_2$ and  holonomy groups $G_1$ and $G_2$, respectively. If $\psi:\widehat{\Gamma}_1\rightarrow\widehat{\Gamma}_2$ is an isomorphism, then there are isomorphisms $\phi:\widehat{M}_1\rightarrow \widehat{M}_2$ and $\varphi: G_1\rightarrow G_2$ such that the following diagram commutes:
		\begin{equation*}\label{diagram 1}
			\begin{CD}
				1 @>{}>> \widehat{M}_1 @>{}>> \widehat{\Gamma}_1 @>{}>> G_1 @>>> 1 \\
				@. @VV{\phi}V @VV{\psi}V @VV{\varphi}V\\
				1 @>{}>> \widehat{M}_2 @>{}>> \widehat{\Gamma}_{2} @>{}>> G_2 @>>> 1.
			\end{CD}
		\end{equation*}
	\end{lem}
	
	\begin{prop}\label{topological invariants}
		Let $\Gamma_1$ and $\Gamma_2$ be $n$-dimensional Bieberbach groups with the cyclic holonomy group $G$ of square-free order $\delta$ and maximal abelian normal subgroups $M$ and $N$, respectively. Then, $\widehat{\Gamma}_1\cong\widehat{\Gamma}_2$ if and only if
		\begin{enumerate}
			\item[(i)] $r(d,M)=r(d,N)$ for each $d\mid \delta$.
			\item[(ii)] $\rho_k(\lambda(s,t,M))=\rho_k(\lambda(s,t,N))$ for each $k=1,\cdots,v$ and $(s,t)\in D_{1}^{\ast}$.
		\end{enumerate}
	\end{prop}
	\begin{proof}
		Suppose $\widehat{\Gamma}_1\cong\widehat{\Gamma}_2$. By Lemma \ref{diagram}, there are isomorphisms $\phi:\widehat{M}_1\rightarrow \widehat{M}_2$ and $\varphi: G\rightarrow G$ such that the following diagram is commutative
		\begin{equation*}
			\begin{CD}
				1 @>{}>> \widehat{M}_1 @>{}>> \widehat{\Gamma}_1 @>{}>> G @>>> 1 \\
				@. @VV{\phi}V @VV{\psi}V @VV{\varphi}V\\
				1 @>{}>> \widehat{M}_2 @>{}>> \widehat{\Gamma}_{2} @>{}>> G @>>> 1.
			\end{CD}
		\end{equation*}
		Then, $\phi:\widehat{M}_1\rightarrow \widehat{M}_2$ is a $\widehat \Z G$-module isomorphism. Now, the statements (i) and (ii) follow immediately from Proposition \ref{profinite version of Th 4.13 of Opp26}.
		
		Conversely assume that $M$ and $N$ are maximal abelian normal subgroups of $\Gamma_1$ and $\Gamma_2$, respectively, satisfying the conditions (i) and (ii). Since $G$ is a cyclic group of square-free order $\delta$, it follows from Lemma \ref{characterization-Bieberbach group with n(G)=1} that $M=M_{n-1}\oplus \Z$ and $N=N_{n-1}\oplus \Z'$ such that $\Gamma_1=M_{n-1}\rtimes C_1$ and $\Gamma_2=N_{n-1}\rtimes C_2$ where $C_1, C_2$ contain $\Z$ and $\Z'$ as subgroups of index $\delta$ and acts on $M_{n-1}, N_{n-1}$ as $G$. Hence, $\widehat{M}=\widehat{M}_{n-1}\oplus\widehat{\Z}, \widehat{N}=\widehat{N}_{n-1}\oplus\widehat{\Z'}$ and $\widehat{\Gamma}_1=\widehat{M}_{n-1}\rtimes\widehat{C}_1, \widehat{\Gamma}_2=\widehat{N}_{n-1}\rtimes\widehat{C}_2$, where $\widehat{C}_1$ and $\widehat{C}_2$ act on $\widehat{M}_{n-1}$ and $\widehat{N}_{n-1}$ as $G$. By Proposition \ref{profinite version of Th 4.13 of Opp26},
		\begin{equation*}
			\widehat{M}_{n-1}\oplus\widehat{\Z} \cong \widehat{N}_{n-1}\oplus\widehat{\Z'}
		\end{equation*}
		as $\widehat{\Z}G$-modules and hence 
		\begin{equation*}
			r(d, M_{n-1})+r(d,\Z)=r(d,N_{n-1})+r(d,\Z')
		\end{equation*}
		and 
		\begin{equation*}
			\rho_k(\lambda(s,t,M_{n-1}))+\rho_k(\lambda(s,t,\Z))=\rho_k(\lambda(s,t,N_{n-1}))+\rho_k(\lambda(s,t,\Z'))
		\end{equation*}
		for each $d\mid \delta$ and $k=1,2,\cdots, v$ by Proposition \ref{thm 7.4 of Opp26}. As $\Z\cong \Z'$ as $\Z G$-modules, by Proposition \ref{Theorem 4.13 of Opp26}  we have 
		\begin{equation*}
			r(d,\Z)=r(d,\Z') \ \ \text{and}  \ \ \rho_k(\lambda(s,t,\Z))=\rho_k(\lambda(s,t,\Z'))
		\end{equation*} 
		for each $d\mid \delta$ and $k=1,2,\cdots,v$. Hence, 
		\begin{equation*}
			r(d,M_{n-1})=r(d,N_{n-1}) \ \ \text{and} \ \ \rho_k(\lambda(s,t,M_{n-1}))=\rho_k(\lambda(s,t,N_{n-1}))
		\end{equation*}
		for each  $k=1,2,\cdots,v$ and $d\mid \delta$, so that by Proposition \ref{profinite version of Th 4.13 of Opp26}, $\widehat{M}_{n-1}\cong\widehat{N}_{n-1}$ as $\widehat{\Z} G$-modules. Hence, $\widehat{\Gamma}_1\cong\widehat{\Gamma}_2$, and the proposition is proved.
	\end{proof}
	We finish this section with the Charlap's classification of Bieberbach groups with the holonomy group of prime order.

	\begin{defn}\label{exceptional Bieberbach group}
		A Bieberbach group $\Gamma$ with prime order holonomy group $C_p$ is exceptional if its maximal abelian normal subgroup $M$  is an exceptional $\Z C_p$-module.
	\end{defn}

	\begin{prop}[\cite{Cha86}, Chapter IV, Theorem 6.3]\label{C1-Bieberbach}
		There is a one-to-one correspondence between the isomorphism classes of non-exceptional Bieberbach groups whose holonomy group has prime order $p$ and $4$-tuples $(a,b,c; \theta)$ where $a,b,c\in \Z$  with $a> 0, \ b\geq 0, \ c\geq 0, \ (a,c)\neq (1,0), \ (b,c)\neq (0,0)$ and $\theta \in\mathrm{Gal}(\zeta_p)\backslash H(\Q(\zeta_p))$. 
	\end{prop}
	
	Note that $\mathrm{Gal}(\zeta_p)$ is cyclic of even order if $p>2$; hence $\mathrm{Gal}(\zeta_p)$ has a subgroup $C_2$ of order $2$ if $p>2$. 
	
	\begin{prop}[\cite{Cha86}, Chapter IV, Theorem 6.4]\label{C2-Bieberbach}
		There is a one-to-one correspondence between the isomorphism classes of exceptional Bieberbach groups whose holonomy group has prime order $p$ and pairs $(b,\theta)$ where $b>0$,  and $\theta \in C_2\backslash H(\Q(\zeta_p))$. 
	\end{prop}
	
	\subsection{Proofs of main results} 
	
	\begin{proof}[Proof of Theorem \ref{cardinality of the crystal class}]
		Let $\Gamma_1$ and $\Gamma_2$ be $n$-dimensional Bieberbach groups  with the cyclic holonomy group of square-free order and maximal abelian normal subgroups $M$ and $N$, respectively. Suppose $\widehat{\Gamma}_1\cong\widehat{\Gamma}_2$. By Lemma \ref{diagram}, we can assume that $\Gamma_1$ and $\Gamma_2$ have the same cyclic holonomy group $G$ of square-free order $\delta$. 
		
		For each $d\mid \delta$, let $C_d$ denote the subgroup of $G$ of order $d$ and let $\Gamma_{i,d}$ denote the $n$-dimensional Bieberbach subgroup of $\Gamma_i$ with the holonomy group $C_d$, for  $i=1,2$. Combining Propositions \ref{topological invariants} and  \ref{profinite version of Th 4.13 of Opp26}, we have
		\begin{eqnarray}\label{equivalence 1}
			\widehat{\Gamma}_1\cong\widehat{\Gamma}_2  &\Leftrightarrow&  \widehat{M}\cong \widehat{N} \ \text{as} \ \nonumber \widehat{\Z}G\text{-modules} \\
			&\Leftrightarrow& \text{for all} \ d\mid \delta, \widehat{M}\cong \widehat{N} \  \text{as} \ \widehat{\Z}C_d\text{-modules} \\
			&\Leftrightarrow&  \text{for all} \ d\mid \delta, \widehat{\Gamma}_{1,d}\cong \widehat{\Gamma}_{2,d}. \nonumber
		\end{eqnarray}
		
		Suppose that there are prime numbers $p$ and $q$ dividing $\delta$ such that  the Bieberbach groups $\Gamma_{1,p}$ and $\Gamma_{2,q}$ are exceptional. Since $\widehat{\Gamma}_1\cong\widehat{\Gamma}_2$, it follows from \eqref{equivalence 1} that $\widehat{\Gamma}_{1,d}\cong\widehat{\Gamma}_{2,d}$ for each $d\mid \delta$. In particular, $\widehat{\Gamma}_{1,p}\cong\widehat{\Gamma}_{2,p}$ and $\widehat{\Gamma}_{1,q}\cong\widehat{\Gamma}_{2,q}$. This implies that, if $p\neq q$, then the Bieberbach groups $\Gamma_{i,p}$ and $\Gamma_{i,q}$ are exceptional, for $i=1,2$. Thus, as $\delta=p_1p_2\cdots p_k$ where $p_l \ (l=1,\cdots,k)$ are distinct prime numbers, we can assume that there is a positive integer $r$ with $1\leq r\leq k$ such that $D=\{q_1,\cdots,q_r\}$ is a subset of $\{p_1,\cdots,p_k\}$ with $r$ distinct  elements, so that the Bieberbach groups $\Gamma_{i,q}$ are exceptional, for each $q\in D$ and $i=1,2$.
		
		We conclude from  Proposition \ref{Cor 27 of DF04} that the Galois group $\mathrm{Gal}(\zeta_\delta)$ has a subgroup $\mathcal{H}_D$ that has $\Z/2\Z$ instead of $\mathrm{Gal}(\zeta_{q})$ as a factor in the direct product for all $q\in D$. Since $\mathrm{Gal}(\zeta_\delta)\cong (\Z/\delta\Z)^{\times}\cong\mathrm{Aut}(G)$, we have $\mathrm{Aut}(G)$ contains an isomorphic copy of $\mathcal{H}_D$, which we will also denote by $\mathcal{H}_D$.
		
		Now, if $[(G,M)]$ denotes the arithmetic crystal class of $\Gamma_1$, then, by Lemma \ref{Invariants_crystal class} and Proposition \ref{Invariants_Semi-Linear Iso}, we have
		\begin{eqnarray*}
			\Gamma_1,\Gamma_2 \in [(G,M)]&\Leftrightarrow& M\cong(N)^{\varphi} \ \text{as} \ \Z G\text{-modules for some} \ \varphi\in \mathcal{H}_D \\
			&\Leftrightarrow& \text{for all} \ d\mid \delta, M\cong(N)^{\varphi} \ \text{as} \ \Z C_d\text{-modules for some} \  \varphi\in \mathcal{H}_D \\
			&\Leftrightarrow& \text{for all} \ d\mid \delta, \Gamma_{1,d},\Gamma_{2,d}\in [(C_d,M)].
		\end{eqnarray*}
		Since the Bieberbach groups $\Gamma_{i,q} \ (q\in D, i=1,2)$ are exceptional, it follows from Propositions \ref{topological invariants} and \ref{C2-Bieberbach} together with  Lemma \ref{Invariants_crystal class} that		
		$$|\mathcal{C}(M)|= \left|\mathcal{H}_D \backslash \prod_{d\mid \delta}H(\Q(\zeta_d))\right|. $$ 
		
		On the contrary, if for each prime $p$ dividing  $\delta$ the Bieberbach groups $\Gamma_{i,p} \ (i=1,2)$ are not exceptional, then by Propositions \ref{topological invariants} and \ref{C1-Bieberbach} together with  Lemma \ref{Invariants_crystal class}, we have 
		$$|\mathcal{C}(M)|= \left|\mathrm{Gal}(\zeta_\delta)\backslash \prod_{d\mid \delta}H(\Q(\zeta_d))\right|, $$ and therefore the Theorem \ref{cardinality of the crystal class} is proved.
		
	\end{proof}
	
	In particular, using Proposition \ref{abstract invariants} and similar arguments as in the proof of Theorem \ref{cardinality of the crystal class}, we get the following generalization of Theorem 1.1 of \cite{Nery}.
	
	\begin{theorem}\label{main Theorem2}
		Let $\Gamma, \Gamma_d$ and $\mathcal{H}_D$ be as in  Theorem \ref{cardinality of the crystal class}. If $\delta=6,10,14,$ or $\delta$ is prime, then
		\begin{equation*}
			|\mathfrak{g}(\Gamma)|= \left|\mathcal{H}_D \backslash \prod_{d\mid \delta}H(\Q(\zeta_d))\right|, 
		\end{equation*}
		for the special case.  Otherwise,
		\begin{equation*}
			|\mathfrak{g}(\Gamma)|= \left|\mathrm{Gal}(\zeta_\delta) \backslash \prod_{d\mid \delta}H(\Q(\zeta_d))\right|. 
		\end{equation*}	
	\end{theorem}
	
	Since the action of the Galois group on ideal class group is transitive if and only if the class group is trivial, we have 
	
	\begin{cor}
		Let $\Gamma$ be an $n$-dimensional Bieberbach group with the cyclic holonomy group of  order $\delta=6,10,14,$ or $\delta$ is prime. Then, $|\mathfrak{g}(\Gamma)|=1$ if and only if for each integer $d$ dividing $\delta$ the class group $H(\Q(\zeta_d))$ is trivial.
	\end{cor}
	
	We can now prove our main theorem.
	
	\begin{proof}[Proof of Theorem \ref{main Theorem}]
		Let $\Gamma$ be an $n$-dimensional Bieberbach group with maximal abelian normal subgroup $M$ and cyclic holonomy group $G$ of square-free order $\delta$. It follows from  Remark \ref{remark-Independence of the CC} that to find all the isomorphism classes of Bieberbach groups of the  arithmetic crystal classes that corresponds to the $\Z G$-lattices in $\mathcal{C}(M)$, it is sufficient to consider a set $T$ of representatives for  the isomorphism classes of $\Z G$-lattices in $\mathcal{C}(M)$. Theorem \ref{cardinality of the crystal class} gives us a formula for the cardinality of $T$. Now, applying first Proposition \ref{Theorem of Mathematical Crystallography} and then Lemma \ref{orbit of N_{GL}(G)}, we have
		\begin{eqnarray*}
			|\mathfrak{g}(\Gamma)|&=&  \sum_{M\in T} |\mathcal{N}_{\mathrm{Aut}(M)}(G)\backslash X(G,M)| \\
			&=& \sum_{M\in T}\left(\prod_{p\mid \delta}|\mathcal{N}_{\mathrm{Aut}(M)}(G)\backslash (\bar{M}_{1,p}^{\ast})^{G}| \right),	
		\end{eqnarray*}
		where $M_{1,p}$ is the largest direct summand of $M$ on which $C_p$ acts trivially.
		
	\end{proof}
	\begin{proof}[Proof of Corollary \ref{cor1}]
		To simplify notation, we let $max\{ |(\bar{M}_{1,p}^{\ast})^{G}| \}$ stand for $max \{ |(\bar{M}_{1,p}^{\ast})^{G}| \  :\ p\mid \delta\}$. By Theorem \ref{main Theorem}, 
		\begin{eqnarray*}
			|\mathfrak{g}(\Gamma)|&=&\sum_{M\in T} \left(  \prod_{p\mid \delta}\left| \mathcal{N}_{\mathrm{Aut}(M)}(G)\backslash (( \bar{M}_{1,p} )^{\ast})^{G} \right|  \right) \\
			&\leq& \sum_{M\in T}\left( \prod_{p\mid\delta} max\{ |(\bar{M}_{1,p}^{\ast})^{G}|\} \right) \\
			&=& \sum_{M\in T} \left(  max\{ |(\bar{M}_{1,p}^{\ast})^{G}| \}  \right)^{b} \\
			&=& |\mathcal{C}(M)| \left(  max\{ |(\bar{M}_{1,p}^{\ast})^{G}| \} \right)^{b}\\
			&\leq& \left( H(\Q(\zeta_{\delta}))\right)^a  \left( max\{ |(\bar{M}_{1,p}^{\ast})^{G}| \} \right)^{b},	\ \  ( \text{by Theorem \ref{cardinality of the crystal class}})
		\end{eqnarray*}
		where $a$ is the number of divisors of $\delta$ and $b$ is the number of prime divisors of $\delta$.
	\end{proof}
	
	\begin{proof}[Proof of Corollary \ref{cor2}]
		This follows from Theorem \ref{main Theorem} together with Lemma \ref{characterization-Bieberbach group with n(G)=1}.
	\end{proof}
	
	\begin{proof}[Proof of Corollary \ref{cor3}]
		This is a consequence of Theorem \ref{main Theorem} together with Lemma \ref{acts transit}.
	\end{proof}

	\subsection{Examples}

	Suppose that $G$ is a cyclic group of square-free order $\delta >1$. Let $\langle a,b \ |\ R \rangle$ be a presentation of  $G$, i.e., $G\cong F_2/\widetilde{R}$ where $F_2$ is the free group on $\{a,b\}$ and $\widetilde{R}$ is the normal closure of $R$ in $F_2$. This defines an exact sequence 
	\begin{equation}\label{seq3}
		1\rightarrow \widetilde{R}\rightarrow F_2\rightarrow G\rightarrow 1.
	\end{equation}
	Hence, if $[\widetilde{R},\widetilde{R}]$ denotes the commutator subgroup of $\widetilde{R}$, then \eqref{seq3} induces  the exact sequence
	\begin{equation*} 
		1\rightarrow M\rightarrow \Gamma\rightarrow G\rightarrow 1,
	\end{equation*} 
	where $M=\widetilde{R}/[\widetilde{R},\widetilde{R}]$ and $\Gamma=F_{2}/[\widetilde{R},\widetilde{R}]$. Note that $M$ is a free abelian group whose rank is given by the Schrier's formula $(\mathrm{rank}(F_2)-1)\delta+1=\delta+1$  (see \cite[Chapter 6, Proposition 2]{Joh97}). Moreover, we claim that $M$ is maximal abelian in $\Gamma$. Suppose the claim were false. Then we could find $x\in F_2$ such that the image of $x$ in $G$ is non-trivial, so that $[x,r]\in [\widetilde{R},\widetilde{R}]$ for all $r\in \widetilde{R}$. Thus, since $G$ is a cyclic group, we have $M$ is in the center of $\Gamma$. Hence, $\Gamma$ is abelian and $[F_2,F_2]=[\widetilde{R},\widetilde{R}]$. As $M$ has finite index in $\Gamma$, we have $\mathrm{rank}(\widetilde{R})=\mathrm{rank}(F_2)$, and hence $|G|=1$, a contradiction. Then, $\Gamma$ is an $(\delta+1)$-dimensional Bieberbach group whose holonomy group is $G$. By \cite[Corollary 1]{Ten97}, $M=\Z G\oplus \Z$ where $\Z$ is a trivial $\Z G$-module.Thus, by Theorem \ref{main Theorem},
	\begin{equation*}
		|\mathfrak{g}(\Gamma)|=	\sum_{M\in T}\left(\prod_{p\mid \delta}|\mathcal{N}_{\mathrm{Aut}(M)}(G)\backslash \mathbb{F}_p^{\ast}| \right),
	\end{equation*}
	since $\Z$ is the largest direct summand of $M$ on which $C_p$ acts trivially for each prime $p\mid \delta$. Since $\Z G$ is a permutation $\Z G$-module (i.e., the $\Z$-basis of $\Z G$ is fixed under the action of $G$), we have $G \leq \mathrm{Sym}(G)\leq\mathrm{GL}(\delta+1,\Z)\cong \mathrm{Aut}(M)$, where $\mathrm{Sym}(G)$ is the group of all permutations of the elements of $G$. Hence, $\mathcal{N}_{\mathrm{Aut}(M)}(G)$ contains the holomorph of $G$  by \cite[Theorem 6.3.2]{Hall59}, i.e., the group $\mathrm{Hol}(G)= G\rtimes \mathrm{Aut}(G)$. As  $\mathrm{Aut}(C_p)\leq \mathrm{Aut}(G)$ for each prime $p\mid \delta$, we have $\mathcal{N}_{\mathrm{Aut}(M)}(G)$ acts transitively on $\mathbb{F}_{p}^{\ast}$ for each prime  $p\mid \delta$. Thus, 
	\begin{equation*}
		|\mathfrak{g}(\Gamma)|=\left|\mathrm{Gal}(\zeta_{d})\backslash \prod_{d\mid \delta} H(\Q(\zeta_d))  \right|, 
	\end{equation*}
	by Theorem \ref{cardinality of the crystal class}. 
	
	In particular, if $\delta=6, 10, 14, 15, 21$ or $\delta$ is prime $\leq 19$, for example,  then $|\mathfrak{g}(\Gamma)|=1$ because for each $d \mid \delta$ the class group $H(\Q(\zeta_{d}))$ is trivial (see \cite[Table 10]{AW04}).
	
	Now, if $\delta=46, 55, 105$ or $\delta$ is prime $> 19$, for example, then $|\mathfrak{g}(\Gamma)|>1$ because $23\mid 46$ and $|H(\Q(\zeta_{46})|\geq 3$; $|H(\Q(\zeta_{55})|\geq 10$; $|H(\Q(\zeta_{105})|\geq 13$; and $|H(\Q(\zeta_{p})|> 1$ for all prime $p> 19$ (see \cite[Theorem 11.1 and p. 353]{Was82}).

	\begin{remark}
		In general, calculating the normalizer of a finite subgroup in $\mathrm{GL}(n,\Z)$ is not easy. We refer the interested reader to \cite{BNZ73}, which presents a method for calculating such normalizers.
	\end{remark}

	\section*{Acknowledgements}
	The author wishes to express his thanks to Prof. Dr. Pavel Zalesskii for his many valuable advices and several discussions.

\end{document}